\documentclass[10pt]{amsart}
\usepackage{amsmath, amscd, amsthm} 
\usepackage{amssymb, amsfonts} 
\usepackage{enumerate}
\usepackage{bbm}
\usepackage{color}
\usepackage[all]{xy} 
\usepackage[margin=1.5in]{geometry}
\usepackage{palatino}
\usepackage{mathrsfs}
\usepackage{booktabs}
\usepackage{aliascnt}
\usepackage{mathrsfs}
\usepackage[svgnames]{xcolor} 
\usepackage{hyperref}
 \definecolor{dark-red}{rgb}{0.4,0.15,0.15}
\hypersetup{
    colorlinks, linkcolor=dark-red,
    citecolor=DarkBlue, urlcolor=MediumBlue
}
\usepackage{mathrsfs}
\usepackage[
textwidth=3cm,
textsize=small,
colorinlistoftodos]
{todonotes}

\usepackage{tikz}

\newcommand{\Q}{\mathbb{Q}} 
\newcommand{\QQ}{\Q}
\newcommand{\CC}{\mathbb{C}} 
\renewcommand{\AA}{\mathbb{A}}

\newcommand{\Z}{\mathbb{Z}}
\newcommand{\ZZ}{\mathbb{Z}}
\newcommand{\R}{\mathbb{R}}
\newcommand{\RR}{\R}

\newcommand{\FF}{\mathbb{F}}
\newcommand{\Ff}{\mathscr{F}}

\newcommand{\iso}{\cong}

\newcommand{\Sm}{\mathrm{Sm}}

\newcommand{\sphere}{\mathbbm{1}}

\newcommand{\SH}{\mathrm{SH}}

\newcommand{\ul}[1]{\underline{\smash{#1}}}

\renewcommand{\setminus}{\smallsetminus}

\DeclareMathOperator*{\colim}{\mathrm{colim}}

\DeclareMathOperator*{\holim}{\mathrm{holim}}

\DeclareMathOperator{\Ext}{Ext}

\renewcommand{\top}{{\mathrm{top}}}

\usepackage[T1]{fontenc}

\numberwithin{equation}{section} 

\theoremstyle{plain}

\newaliascnt{theorem}{equation}  
\newtheorem{theorem}[theorem]{Theorem}  
\aliascntresetthe{theorem}

 \theoremstyle{definition}

\newaliascnt{prop}{equation}  
\newtheorem{prop}[prop]{Proposition}
\aliascntresetthe{prop}

\newaliascnt{lemma}{equation}  
\newtheorem{lemma}[lemma]{Lemma}
\aliascntresetthe{lemma}

\newaliascnt{corollary}{equation}  
\newtheorem{corollary}[corollary]{Corollary}
\aliascntresetthe{corollary}

\newaliascnt{claim}{equation}  

\aliascntresetthe{claim}

\newaliascnt{conjecture}{equation}  

\aliascntresetthe{conjecture}

\newaliascnt{question}{equation}  
\newtheorem{question}[question]{Question}
\aliascntresetthe{question}

\newaliascnt{defn}{equation}  

\aliascntresetthe{defn}

\newaliascnt{example}{equation}  
\newtheorem{example}[example]{Example}
\aliascntresetthe{example}

\theoremstyle{remark}

\newaliascnt{remark}{equation}  
\newtheorem{remark}[remark]{Remark}
\aliascntresetthe{remark}

\newaliascnt{convention}{equation}  
\newtheorem{convention}[convention]{Convention}
\aliascntresetthe{convention}

\newcommand{\aref}[1]{\autoref{#1}}

\renewcommand{\AA}{\mathbb{A}}
\newcommand{\SHA}{\mathrm{SH}^{\AA^1}\!}

\newcommand{\PP}{\mathbb{P}}

\newcommand{\Spec}{\operatorname{Spec}}

\newcommand{\chr}{\operatorname{char}}
\newcommand{\cd}{\operatorname{cd}}
\newcommand{\vcd}{\operatorname{vcd}}
\newcommand{\comp}[1]{^{\widehat{~}}_{#1}}
\newcommand{\ret}{\text{r\'et}}
\newcommand{\MU}{\mathrm{MU}}
\newcommand{\EM}{\mathbf{M}}
\newcommand{\BP}{\mathrm{BP}}
\newcommand{\Shv}{\operatorname{Shv}}

\begin{document}
\title{Vanishing in stable motivic homotopy sheaves}
\author{Kyle Ormsby}
\address{Reed College}
\email{ormsbyk@reed.edu}
\author{Oliver R\"{o}ndigs}
\address{Universit\"{a}t Osnabr\"{u}ck}
\email{oroendig@uni-osnabrueck.de}
\author{Paul Arne {\O}stv{\ae}r}
\address{University of Oslo}
\email{paularne@math.uio.no}
\begin{abstract}
We determine systematic regions in which the bigraded homotopy sheaves of the motivic sphere spectrum vanish.
\end{abstract}

\maketitle

\section{Introduction}\label{sec:intro}
Stable motivic homotopy theory is a nonabelian generalization of the algebro-geometric theory of motives.  It is a natural arena in which to study the motivic cohomology, $K$-theory, and algebraic cobordism of smooth schemes \cite{v:icm}, and its invention was crucial to the resolution of the Milnor and Bloch-Kato conjectures \cite{v:2,v:l}.

The most fundamental objects in the stable motivic homotopy category $\SHA(F)$ (over a field $F$) are the $\PP^1$-suspension spectra $\Sigma^\infty_{\PP^1}U_+$ for $U$ a smooth $F$-scheme.  Distinguished amongst these is the sphere spectrum $\sphere := \Sigma^{\infty}_{\PP^1}\Spec(F)_+$.  We denote this object by $\sphere$ because it is the unit for the symmetric monoidal product on $\SHA(F)$ given by the smash product.

Equivalence between $\PP^1$-spectra is detected by the bigraded homotopy sheaves, $\ul\pi_{m+n\alpha}X$ for $m,n\in\ZZ$, which are defined as the Nisnevich sheafification of the assignment
\[
 U\in \Sm/\Spec(F)\longmapsto [\Sigma^{m+n\alpha} U_+,X]_{\SHA(F)}.
\]
Here $[~,~]_{\SHA(F)}$ denotes the hom-set in $\SHA(F)$ and $\Sigma^{m+n\alpha}$ denotes smashing with $(S^1)^{\wedge m}\wedge(\AA^1\smallsetminus 0)^{\wedge n}$.  Since every motivic spectrum is a $\sphere$-module, the bigraded sheaf
\[
  \ul\pi_\star \sphere = \bigoplus_{m,n\in \ZZ}\ul\pi_{m+n\alpha}\sphere
\]
plays a fundamental role in stable motivic homotopy theory, analogous to the stable homotopy groups of spheres in topology.  We will refer to $\ul\pi_{m+n\alpha}\sphere$ as the \emph{$(m+n\alpha)$-th motivic stable stem}, and to the $\ZZ$-graded sheaf $\ul\pi_{m+*\alpha}\sphere$ as the \emph{$m$-th Milnor-Witt stem}.

The motivic stable stems (and their global sections, $\pi_{m+n\alpha}\sphere := \ul\pi_{m+n\alpha}\sphere(F)$) have been objects of intense study since Morel's analysis of the $0$-th motivic stable stem in \cite{morel:suite}.  That paper launched his program \cite{morel:pi0} to identify the $0$-th Milnor-Witt stem with $\ul K^{MW}_{-*}$, the Milnor-Witt $K$-theory sheaf, explaining the nomenclature.  In further work \cite{morel:conn}, Morel shows that $\sphere$ is connective, meaning that $m$-th Milnor-Witt stems are $0$ for $m<0$.

Beyond Morel's theorems, little is known about Milnor-Witt stems over a general field.  R\"ondigs-Spitzweck-{\O}stv{\ae}r \cite{RSO:pi1} determine the $1$-st Milnor-Witt stem as an extension of $\ul K^M_*/24$ and a certain sheaf related to Hermitian $K$-theory; this vastly generalizes work of Ormsby-{\O}stv{\ae}r \cite{MR3255457} for fields of cohomological dimension less than three.
All other computations are limited to specific fields, and are generally only known on global sections (and frequently after completion at $2$).  Indeed, Hu-Kriz-Ormsby \cite{hko:C} and Dugger-Isaksen \cite{di:mass} make computations over $\CC$ via the Adams-Novikov and Adams spectral sequences, Ormsby \cite{o:thesis} makes computations over $p$-adic fields, Heller-Ormsby \cite{ho:e2m, ho:e2m2} and Dugger-Isaksen \cite{di:RC2,di:lowR} make computations over $\RR$, 
and Wilson-{\O}stv{\ae}r \cite{WO:ff} over finite fields. All these computations hold only in specific (often finite) ranges.

In this paper, we exploit the methods of \cite{RSO:pi1} to find conditions under which the $m$-th Milnor-Witt stem is bounded above; see \aref{thm:integral0} and \aref{thm:integral1}.  Our methods apply to a general field $F$ of characteristic different from $2$, and they result in sheaf level theorems (after inverting the exponential characteristic of $F$).\footnote{So $q=1$ if $\chr F = 0$ and otherwise $q=\chr F$.}

Our vanishing theorems have important implications for the nonzero homotopy sheaves of $\sphere$ via Morel's contraction construction \cite{morel:A1}.  Given a Nisnevich sheaf of abelian groups $\Ff$ on $\Sm/F$, the \emph{contraction} $\omega\Ff$ of $\Ff$ takes $U$ to the kernel of $\Ff(U\times (\AA^1\smallsetminus 0))\to \Ff(U)$.  (Here the map is induced by the canonical section $1:\Spec F\to \AA^1\smallsetminus 0$.)  For any motivic spectrum $E$, we have $\omega\ul\pi_{m+n\alpha}E \cong \ul\pi_{m+(n+1)\alpha}E$.  In particular, if $\ul\pi_{m+n\alpha}E = 0$, then, for $k\ge 1$, the $k$-fold contraction of $\ul\pi_{m+(n-k)\alpha}E$ is $0$.  Future computations should be able to exploit vanishing of $\ul\pi_{m+n\alpha}\sphere$ to constrain the structure of $\ul\pi_{m+\ell\alpha}\sphere$ for $\ell<n$.

We now state our results precisely, giving some indication of our methods along the way.  Fix a field $F$ and let $q$ denote its exponential characteristic.  We begin by studying the $\eta$-complete sphere spectrum via Voevodsky's slice spectral sequence \cite{v:newopen} using the results in \cite{RSO:pi1}, and then ``un-complete'' our results via a sequence of fracture squares.

Let $\eta\in\ul\pi_\alpha \sphere(\Spec F)$ denote the motivic Hopf map induced by the projection $\AA^2\smallsetminus 0\to \PP^1$.  The $\eta$-complete sphere is the motivic spectrum $\hat \sphere = \holim_n \sphere/\eta^n$.  Our first result is the following.

\begin{theorem}\label{thm:main}
Over a field $F$ with exponential characteristic $q\ne 2$,
\[
  \ul\pi_{m+n\alpha}\hat \sphere[1/q] = 0
\]
whenever
\begin{itemize}
\item $m<0$, or
\item $m>0$, $m\equiv 1$ or $2\pmod 4$, and $2n>\max\{3m+5,4m\}$.
\end{itemize}
\end{theorem}

Using the same techniques which prove \aref{thm:main}, we can prove a stronger vanishing result for $\hat \sphere_{(p)}$, the $\eta$-completion of the $p$-local sphere spectrum.  (The notation $\widehat{\sphere_{(p)}}$ might be more appropriate, but we find it unwieldy.)

\begin{theorem}\label{thm:plocal}
Let $F$ be a field and let $p$ be an odd prime different from the characteristic of $F$.  Then
\[
  \ul\pi_{m+n\alpha}\hat \sphere_{(p)} = 0
\]
whenever
\begin{itemize}
\item $m<0$, or
\item $m\ge 0$ and $(p-2)n>(p-1)m$.
\end{itemize}
\end{theorem}

While the $\eta$-complete sphere is an interesting object in its own right, one would like to know if there are vanishing regions in $\ul\pi_\star \sphere$ as well.  When the cohomological dimension $\cd F <\infty$, this is known by the following theorem, essentially due to Levine \cite{Levine:convergence}.\footnote{In \cite{Levine:convergence}, Levine proves that the slice spectral sequence for $\sphere$ converges to $\ul\pi_\star\sphere$ when $\cd F<\infty$.  Meanwhile, \cite[Theorem 3.50]{RSO:pi1} identifies the target with $\ul\pi_\star\hat\sphere$, so the completion map $\sphere\to\hat\sphere$ induces an isomorphism $\ul\pi_\star\sphere\cong \ul\pi_\star\hat\sphere$ and thus $\sphere \simeq \hat\sphere$.}

\begin{theorem}\label{thm:cdfinite}
Suppose $F$ is perfect and $\cd F <\infty$.  Then $\sphere\simeq \hat\sphere$.\hfill\qedsymbol
\end{theorem}

From this, we can deduce the following vanishing theorem.

\begin{theorem}\label{thm:integral0}
Suppose $F$ is nonreal with exponential characteristic $q\ne 2$.  If $q>2$, further suppose that $F$ is perfect and of finite cohomological dimension.  Then $\ul\pi_{m+n\alpha}\sphere[1/q]$ vanishes in the range given in \aref{thm:main}.  If $p\ne q$ is an odd prime, then $\ul\pi_{m+n\alpha}\sphere_{(p)}$ vanishes in the range given in \aref{thm:plocal}.
\end{theorem}

The positive characteristic statement is a direct consequence of \aref{thm:cdfinite}, but the characteristic $0$ nonreal case does not exclude the possibility of infinite cohomological dimension.  We handle this via standard base change methods which we explain in \aref{subsec:base}.

When $F$ is formally real, vanishing in $\ul\pi_\star\sphere$ is more interesting.  We first observe that the $\eta$-primary fracture square
\[\xymatrix{
  \sphere\ar[r]\ar[d] &\hat\sphere\ar[d]\\
  \eta^{-1}\sphere\ar[r] &\eta^{-1}\hat\sphere
}\]
reduces the problem to that of vanishing regions for $\ul\pi_\star \eta^{-1}\sphere$.  We solve this problem using Bachmann's theorem on $\ul\pi_\star \sphere[1/2,1/\eta]$ and the Hu-Kriz-Ormsby comparison of the $2$- and $(2,\eta)$-complete spheres when $F$ has finite virtual $2$-cohomological dimension.\footnote{Recall that $\vcd_2(F):=\cd_2(F(\sqrt{-1}))$ where $\cd_2$ denotes $2$-primary \'etale cohomological dimension.}  In order to state our results, let $\pi_m^\top \sphere$ denote the $m$-th homotopy group of the topological sphere spectrum.

\begin{theorem}\label{thm:integral1}
Suppose $F$ is formally real.  If $m<0$, $\ul\pi_\star \sphere=\ul\pi_\star\sphere_{(p)} = 0$ by Morel's connectivity theorem.  Suppose $m>0$.  Then $\ul\pi_{m+n\alpha}\sphere = 0$ whenever $\ul\pi_{m+n\alpha}\hat \sphere = 0$ (see \aref{thm:main}) and $\pi_m^\top \sphere[1/2] = 0$.  If $p$ is an odd prime, then $\ul\pi_{m+n\alpha}\sphere_{(p)} = 0$ whenever $\ul\pi_{m+n\alpha}\hat \sphere_{(p)} = 0$ (see \aref{thm:plocal}) and $\pi_m^\top \sphere_{(p)} = 0$.
\end{theorem}

Of course, determining when $\pi_m^\top \sphere[1/2]$ or $\pi_m^\top \sphere_{(p)}$ is $0$ is no easy task.  Nonetheless, we view these conditions as ``reductions to topology'' which effectively transfer the problem from motivic to classical homotopy theory.  Given the intractability of these topological vanishing problems, this is the best type of result we can hope for.

\begin{example}\label{ex:explicit-vanishing}
  Suppose $F$ is formally real.  Toda's calculations say that $\pi_{18}^\top \sphere = \ZZ/8\oplus \ZZ/2$ \cite[p.188]{toda}.
  Since $18\equiv 2 \pmod 4$ and $2\cdot 37>4\cdot 18$, \aref{thm:integral1} implies that $\ul\pi_{18+37\alpha}\sphere = 0$. 
  Ravenel's calculations imply that $\pi_{61}^\top \sphere[1/2] = 0$ \cite[Theorem 1.1.13, A3.4, A3.5, Theorem 4.4.20]{ravenel}. Since $61\equiv 1 \pmod 4$
  and $2\cdot 123 > 4\cdot 61$, one obtains $\ul\pi_{61+123\alpha}\sphere =0$.
\end{example}

\begin{convention}
Henceforth, we always invert the exponential characteristic $q$ of the base field $F$, but we omit this from our notation.  Note that when $F$ is formally real, $q=1$, so our theorems for formally real fields are genuinely integral.
\end{convention}

\subsection*{Outline}
In \aref{sec:prelim}, we collect some necessary facts about the slice and (Adams-)Novikov spectral sequences.  In \aref{sec:eta}, we use the slice spectral sequence to prove \aref{thm:main} and \aref{thm:plocal}.  \aref{sec:uncomp} is split into three subsections.  In \aref{subsec:base}, we review some base change theorems and use them to prove \aref{thm:integral0}.  In \aref{subsec:bachmann}, we review Bachmann's theorem on $\SHA(F)[1/2,1/\eta]$.  In \aref{subsec:integral1}, we use fracture square methods to prove \aref{thm:integral1}.  Finally, in \aref{sec:q}, we pose several open questions in \aref{sec:q} and prove \aref{thm:etainv} which compares the motivic stable stems and $\eta$-periodic motivic stable stems in a range.

\section{Preliminaries}\label{sec:prelim}

In this section, we gather known facts about the slice and Novikov spectral sequences that we will need for our arguments.

We use the slice spectral sequence to prove \aref{thm:main} and \aref{thm:plocal}.  See \cite{v:newopen} for its construction and \cite{RSO:pi1} for a contemporary take on its properties.  The slice spectral sequence for the sphere takes the form
\[
  E_1^{m,n,t} = \ul\pi_{m+n\alpha}s_t \sphere\implies \ul\pi_{m+n\alpha}\hat \sphere
\]
where $s_t\sphere$ is the $t$-th slice of the sphere spectrum \cite[Theorem 3.50]{RSO:pi1}.  By \emph{loc.~cit.}~applied to the admissible pair $(\Spec F, \ZZ_{(p)})$, its $p$-local analogue takes the form
\[
  E_1^{m,n,t}(p) = \ul\pi_{m+n\alpha}s_t\sphere_{(p)}\implies \ul\pi_{m+n\alpha}\hat\sphere_{(p)}
\]
whenever $p\ne \chr F$.  In the case of \aref{thm:plocal}, we prove that $E_1^{m,n,t}(p)$ vanishes in the stated range, implying that $\ul\pi_\star \hat \sphere_{(p)}$ vanishes in the same range.  In the case of \aref{thm:main}, we must work a little harder and show that $E_2^{m,n,t}$ vanishes in appropriate regions.

Both proofs depend crucially on the form which the slices of $\sphere$ take.  Surprisingly, these slices are governed by the $E_2$-page of the Novikov (\emph{i.e.,} $\MU$-Adams) spectral sequence from classical stable homotopy theory.  Let 
\[
  E_2^{s,t}(\MU) = \Ext_{\MU_*\MU}^{s,t}(\MU_*,\MU_*)
\]
denote the cohomology of the $\MU$ Hopf algebroid where $s$ denotes homological degree and $t$ the internal grading on $\MU_*$.  (Note that $\MU_*$ is even-graded, so this group vanishes whenever $t$ is odd.)  Furthermore, let $\EM$ denote the motivic Eilenberg-MacLane functor which takes in an abelian group $A$ and produces the spectrum $\EM A$ representing motivic cohomology with coefficients in $A$.  Using this notation, we get the following theorem due to R\"ondigs-Spitzweck-{\O}stv{\ae}r.

\begin{theorem}[{\cite[Theorem 2.12]{RSO:pi1}}]\label{thm:slice}
The $t$-th slice of the motivic sphere spectrum is
\[
  s_t\sphere = \bigvee_{s\ge 0}\Sigma^{t-s+t\alpha}\EM E_2^{s,2t}(\MU).
\]
\hfill \qedsymbol
\end{theorem}

To further understand the $E_1$-page of the slice spectral sequence, we will need more information on two things:  first, the homotopy sheaves of motivic Eilenberg-MacLane spectra, and second, structural properties of $E_2^{s,t}(\MU)$.

\begin{lemma}[{\cite[Corollary 3.2.1]{SV}}]\label{lemma:EM}
Suppose $A$ is a finitely generated abelian group.  Then
\[
  \ul\pi_{m+n\alpha}\EM A = 0
\]
for $m<0$, or $m=0$ and $n>0$, or $m>0$ and $n>-1$.
\end{lemma}
\begin{proof}
We have
\[
\begin{aligned}
  (\ul\pi_{m+n\alpha} \EM A)(U) &= [S^{m+n\alpha}\wedge U_+,\EM A]\\
  &\cong [U_+,S^{-m-n\alpha}\wedge \EM A]\\
  &\cong H^{-m-n}(U;A(-n)).
\end{aligned}
\]
The stated vanishing range then follows from \cite[Corollary 3.2.1]{SV}.
\end{proof}

While discussing motivic Eilenberg-MacLane spectra, we take a moment to note the following lemma which we will need later in our arguments.

\begin{lemma}\label{lemma:tau}
Suppose $\chr F\ne 2$.  Then multiplication by $\tau$ is injective on $\ul\pi_\star\EM\FF_2$.
\end{lemma}
\begin{proof}
The spectrum $\EM\FF_2$ is cellular by \cite[Proposition 8.1]{Hoyois}.
Thus we can check injectivity of  $\tau:\ul\pi_\star\EM\FF_2\to \ul\pi_{\star+1-\alpha}\EM\FF_2$ on homotopy \emph{groups}, \emph{i.e.,} after evaluating the sheaves at $\Spec F$.  
By the solution of the Milnor conjecture, $\ul\pi_\star\EM\FF_2(\Spec F)\cong K^M_*(F)/2[\tau]$, so injectivity of $\tau$ is obvious.
\end{proof}

We now turn to the structure of $E_2^{s,t}(\MU)$.  
This has been an object of intense study since the 1970s, and the results we need are easily culled from the literature.  
We first consider various finiteness properties, and how to build up $E_2^{s,t}(\MU)$ from $p$-local information.

\begin{lemma}\label{lemma:BPMU}
\begin{enumerate}[(a)]
\item Unless $(s,t) = (0,0)$, the group $E_2^{s,t}(\MU)$ is finite; furthermore, $E_2^{0,0}(MU) = \ZZ$.
\item Let $\mathscr{P}$ denote the set of rational primes and for $p\in \mathscr{P}$ let $\BP(p)$ denote the $p$-local Brown-Peterson spectrum.  Let $E_2^{*,*}(\BP(p))$ denote the cohomology of the $\BP(p)$ Hopf algebroid.  Then
\[
  E_2^{>0,*}(\MU) \cong \bigoplus_{p\in\mathscr{P}}E_2^{>0,*}(\BP(p)).
\]
\item There is a vanishing line so that $E_2^{s,t}(\BP(p)) = 0$ when $t<2s(p-1)$.
\end{enumerate}
\end{lemma}
\begin{proof}
These are all standard results going back to Novikov and Zahler.  For (a), see \cite[Proposition 2.1]{Novikov}.  For (b), see \cite[p.482]{Zahler}.  For (c), see \cite[Corollary 3.1]{Novikov}.
\end{proof}
\begin{remark}
Readers expert with the Adams-Novikov spectral sequence will notice that we did not include the sparsity theorem $E_2^{s,t}(\BP(p)) = 0$ whenever $2p-2\nmid t$.  Combined with \aref{thm:slice}, sparsity certainly gives interesting information about the suspension bigrading of $p$-local slice summands.  But because of the fourth quadrant cone worth of nonzero homotopy sheaves associated with an Eilenberg-MacLane spectrum (\aref{lemma:EM}), we do not get analogous sparsity results on $\ul\pi_\star \hat \sphere_{(p)}$, at least when $\cd F = \infty$.  If $\cd F$ is finite and sufficiently small relative to $p$, then one can deduce a sort of sparsity result for $\ul\pi_\star \sphere$, but we do not pursue the specifics here.
\end{remark}

Finally, we will need to leverage the Andrews-Miller analysis of the $\alpha_1$-inverted $2$-local Adams-Novikov spectral sequence \cite{AM}.  Recall that $\alpha_1$ is the generator of $E_2^{1,2}(\BP(2))\cong E_2^{1,2}(\MU)$.

\begin{lemma}\label{lemma:AM}
There is an isomorphism
\[
  \alpha_1^{-1}E_2^{*,*}(\MU)\cong \FF_2[\alpha_1^{\pm 1},\alpha_3,\alpha_4]/(\alpha_4^2)
\]
where $|\alpha_3| = (1,6)$ and $|\alpha_4| = (1,8)$.  Moreover, the localization map \[E_2^{s,t}(\MU)\to \alpha_1^{-1}E_2^{s,t}(\MU)\] is an isomorphism whenever $t<6s-10$ and $t<4s$.
\end{lemma}
\begin{proof}
In \cite[Corollary 6.2.3]{AM}, Andrews and Miller prove that
\[
  \alpha_1^{-1}E_2^{*,*}(\BP(2))\cong \FF_2[\alpha_1^{\pm 1},\alpha_3,\alpha_4]/(\alpha_4^2).
\]
Since $2\alpha_1=0$, the $E_2^{*,*}(\MU)$ version of this isomorphism follows from \aref{lemma:BPMU}(b).

For the second part of the lemma, note that by \cite[Proposition 5.1]{AM}, \[E_2^{s,t}(\BP(2))\to \alpha_1^{-1}E_2^{s,t}(\BP(2))\] is an isomorphism when $t<6s-10$.  For $p>2$, \aref{lemma:BPMU}(c) implies that $E_2^{s,t}(\BP(p)) = 0$ for $t<2s(p-1)\le 4s$.  Hence when $t<\min\{6s-10,4s\}$, we are guaranteed to only have $2$-primary groups in the Andrews-Miller isomorphism range.
\end{proof}

\begin{figure}
\begin{center}
\begin{tikzpicture}[x=.6cm,y=.6cm]
\draw[step=1.0,gray,ultra thin] (-.9,-.9) grid (19.9,7.9);
\draw[thick] (0,-.9) -- (0,7.9) node[anchor = south east] {$s$};
\draw[thick] (-.9,0) -- (19.9,0) node[anchor = north west] {$t-s$};

\draw[thick,red] (0,0) -- (7.9,7.9) node[anchor = south west] {$p=2$};
\draw[thick,green] (0,0) -- (19.9,19.9/3) node[anchor = north west] {$p=3$};
\draw[thick,blue] (0,0) -- (19.9,19.9/5) node[anchor = west] {$p=5$};
\draw[thick,purple] (0,0) -- (19.9,19.9/7) node[anchor = west] {$p=7$};
\draw[thick,magenta] (0,0) -- (19.9,19.9/11) node[anchor = west] {$p=11$};

\draw[thick,cyan] (0,2) -- (10,4) -- (15,5.05);
\draw[thick,cyan] (15,5.05) -- (19.9,19.9/3+.05) node[anchor = south west] {A-M};

\end{tikzpicture}
\caption{This diagram represents structural features of $E_2^{s,t}(\MU)$ and $E_2^{s,t}(\BP(p))$ for various primes $p$.  Note that we have drawn the $E_2$-page in the tradition Adams grading with $t-s$ on the horizontal axis and $s$ on the vertical axis.  The lines labelled by a prime $p$ correspond to $p$-local vanishing lines, so that $E_2^{s,t}(\BP(p)) = 0$ above these lines.  The piecewise linear curve corresponds to the Andrews-Miller range of \aref{lemma:AM}.}\label{fig:Novikov}
\end{center}
\end{figure}
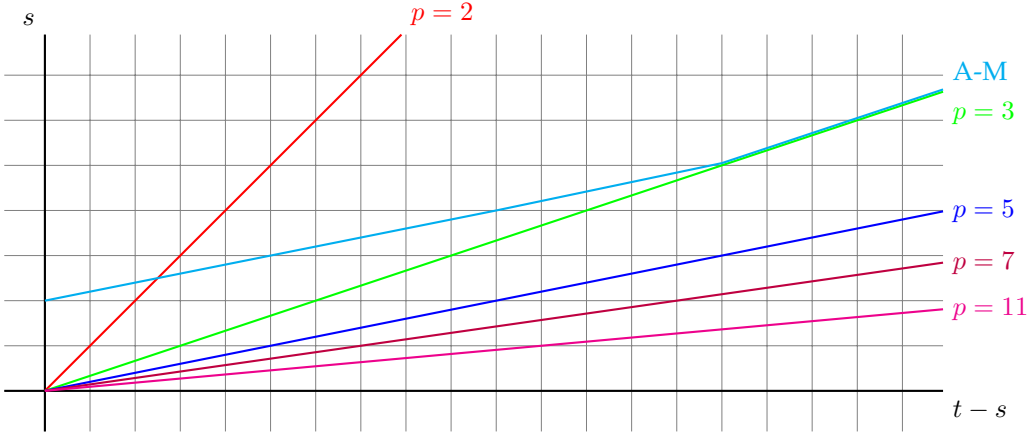

\section{Vanishing for the $\eta$-complete sphere}\label{sec:eta}

This section consists of the proofs of \aref{thm:main} and \aref{thm:plocal}.

\begin{proof}[Proof of \aref{thm:main}]
Vanishing for $m<0$ follows from the vanishing range in the $E_1$-page of the slice spectral sequence, 
which in turn relies on Morel's connectivity theorem \cite{morel:connectivity}.  
We turn to the second condition, namely vanishing of $\ul\pi_{m+n\alpha}\hat \sphere$ when $m>0$, $m\equiv 1\text{ or }2\pmod 4$, and $2n>\max\{3m+5,4m\}$.  Recall that the slice spectral sequence
\[
  E_1^{m,n,t} = \ul\pi_{m+n\alpha}s_t \sphere\implies \ul\pi_{m+n\alpha} \hat \sphere
\]
converges to the homotopy sheaves of $\hat \sphere$.  By \aref{thm:slice}, we may rewrite the $E_1$-page as
\[
  E_1^{m,n,t} = \bigoplus_{s\ge 0} \ul\pi_{m+n\alpha}\Sigma^{t-s+t\alpha}\EM E_2^{s,2t}(\MU).
\]

Let $T$ denote the linear transformation of the $(s,t)$-plane to the $(m+n\alpha)$-plane given by the matrix $\begin{pmatrix}-1&1/2\\ 0&1/2\end{pmatrix}$ (where both planes are given their standard bases).  We call $T$ the \emph{Novikov-to-slice grading shift} since Novikov $E_2$-terms in degree $(s,t)$ correspond to slice summands which are Eilenberg-MacLane spectra shifted by $T(s,t)$.  We say that a bigrading $m+n\alpha$ \emph{contains a slice summand} if there exist integers $s$, $t$ such that $t-s+t\alpha = m+n\alpha$ and $E_2^{s,2t}(\MU)\ne 0$.  By grade school algebra and \aref{thm:slice}, $m+n\alpha$ contains a slice summand if and only if $E_2^{n-m,2n}(\MU)\ne 0$.  We see then that under the Novikov-to-slice grading shift $T$, vanishing regions in $E_2^{s,t}(\MU)$ are mapped to bigradings which do not contain a slice summand.

Similarly, other structural properties are preserved by $T$ as long as statements are translated into the language of slice summands (rather than groups or sheaves); see \aref{fig:Novikov} and \aref{fig:slice}.  For instance, let $\boldsymbol{\alpha}$ denote the region in the $(s,t)$-plane specified in \aref{lemma:AM} in which the map $E_2^{s,t}(\MU)\to \alpha_1^{-1}E_2^{s,t}(\MU)$ is an isomorphism; \emph{i.e.,}
\[
  \boldsymbol{\alpha} = \{(s,t)\mid t<\min\{6s-10,4s\}\}.
\]
Then within the region
\[
  T(\boldsymbol{\alpha}) = \{m+n\alpha\mid 2n>\max\{3m+5,4m\}\},
\]
we know that every nontrivial slice summand is a suspension of $\EM \FF_2$ indexed by a monomial of the form $\alpha_1^i\alpha_3^j\alpha_4^\varepsilon$ where $i$ and $j$ are sufficiently large integers and $\varepsilon = 0$ or $1$.  The suspension bigrading for an $\EM \FF_2$ indexed by $\alpha_1^i\alpha_3^j\alpha_4^\varepsilon$ is $(2j+3\varepsilon) + (i+3j+4\varepsilon)\alpha$.

\begin{figure}
\begin{center}
\begin{tikzpicture}[x=.8cm,y=.4cm]
\draw[step=1.0,gray,ultra thin] (-.9,-.9) grid (7.9,15.9);
\draw[thick] (0,-.9) -- (0,15.9) node[anchor = south east] {$n\alpha$};
\draw[thick] (-.9,0) -- (7.9,0) node[anchor = north west] {$m$};

\draw[thick,red] (-.05,.0) -- (-.05,15.9) node[anchor = north east] {$p=2$};
\draw[thick,green] (0,0) -- (7.9, 2*7.9) node[anchor = north west] {$p=3$};
\draw[thick,blue] (0,0) -- (7.9,7.9*4/3) node[anchor = west] {$p=5$};
\draw[thick,purple] (0,0) -- (7.9,7.9*6/5) node[anchor = west] {$p=7$};
\draw[thick,magenta] (0,0) -- (7.9,7.9*10/9) node[anchor = west] {$p=11$};
\draw[thick,darkgray] (0,0) -- (7.9,7.9) node[anchor = west] {$p\to \infty$};

\draw[thick,cyan] (0,1.5) -- (5,10.05) -- (7.9,7.9*2+.05) node[anchor = south west] {A-M};

\end{tikzpicture}
\end{center}
\caption{This diagram represents structural features of the $E_1$- and $E_2$-pages of the slice spectral sequence for the motivic sphere spectrum.  The $m$- and $n\alpha$-axes are presented, with the slice grading $t$ suppressed.  After the Novikov-to-slice grading shift, the $p$-local vanishing lines of \aref{fig:Novikov} become the vanishing lines of \aref{thm:plocal}, and the Andrews-Miller vanishing curve gives the vanishing curve of \aref{thm:main} (for $m\equiv 1\text{ or }2\pmod{4}$).  Note that the vanishing range given by \aref{thm:main} is much better than the naive range given by the $p=2$ case of \aref{thm:plocal}.}\label{fig:slice}
\end{figure}
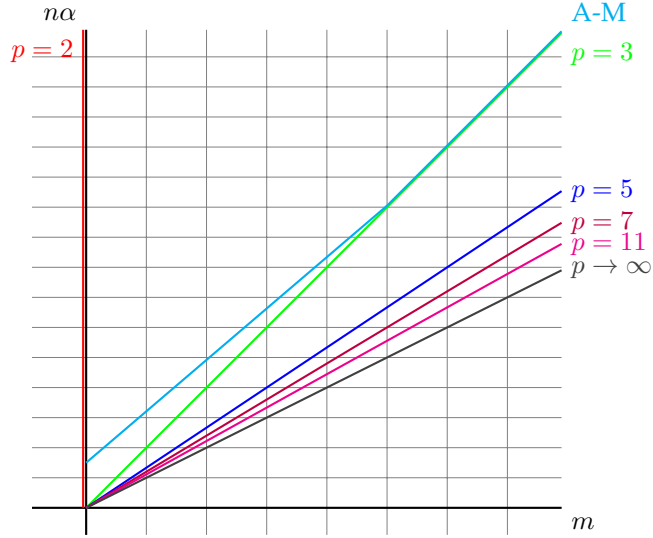

We now show that in the slice spectral sequence we have
\[
  E_2^{m,n,t} = 0
\]
when $m+n\alpha\in T(\boldsymbol\alpha)$ and $m\equiv 1$ or $2\pmod{4}$.  By \cite[Lemma 4.2]{RSO:pi1}, there are $d_1$ differentials in the slice spectral sequence which restrict to
\[
  \tau \mathrm{pr}:\Sigma^{4q+1+(4q+2)\alpha}\EM\ZZ/a_{2q}\ZZ\to \Sigma^{4q+(4q+3)\alpha}\EM\FF_2
\]
on $\overline{\alpha}_{4q+2}$,
\[
  \tau:\Sigma^{4q+1+(4q+2+j)\alpha}\EM\FF_2\to \Sigma^{4q+(4q+3+j)\alpha}\EM\FF_2
\]
on $\alpha_1^j\alpha_{4q+2}$ for $j\ge 1$, and
\[
 \tau:\Sigma^{4q-2+(4q-1+j)\alpha}\EM\FF_2\to \Sigma^{4q-3+(4q+j)\alpha}\EM\FF_2
\]
on $\alpha_1^j\alpha_{4q-1}$ for $j\ge 0$.  Within $T(\boldsymbol\alpha)$, these differentials are multiplication by $\tau$ linking suspended $\EM\FF_2$ summands indexed by $\alpha_1^i\alpha_3^{2j+1}\alpha_4^\varepsilon$ to suspended $\EM\FF_2$ summands indexed by $\alpha_1^{4+i}\alpha_3^{2j}\alpha_4^\varepsilon$.  By \aref{lemma:tau}, multiplication by $\tau$ is injective on $\ul\pi_\star \EM\FF_2$.  It follows that $E_2^{T(\boldsymbol\alpha),*}$ is concentrated in columns indexed by $m\equiv 0\text{ or }3\pmod{4}$.

This gives our desired vanishing result on the slice $E_2$-page, which in turn implies that $\ul\pi_{m+n\alpha}\hat\sphere = 0$ when $m+n\alpha\in T(\boldsymbol\alpha)$ and $m\equiv 1\text{ or }2\pmod{4}$, concluding our proof.
\end{proof}

\begin{proof}[Proof of \aref{thm:plocal}]
Consider the $p$-local slice spectral sequence
\[
  E_1^{m,n,t}(p) = \ul\pi_{m+n\alpha}s_t\sphere_{(p)}\implies \ul\pi_{m+n\alpha}\hat\sphere_{(p)}.
\]
A $p$-local version of \aref{thm:slice} implies that
\[
  s_t \sphere_{(p)} = \bigvee_{s\ge 0}\Sigma^{t-s+t\alpha}\EM E_2^{s,2t}(\BP(p)).
\]
By \aref{lemma:BPMU}(c), $E_2^{s,t}(\BP(p))= 0$ for $t<2s(p-1)$.  Under the Novikov-to-slice  grading shift $T$ (see the proof of \aref{thm:main}), this region becomes
\[
  \{m+n\alpha\mid (p-2)n>(p-1)m\}.
\]
As such,
\[
  s_n \sphere_{(p)} \simeq *
\]
for $(p-2)n>(p-1)m$.  By \aref{lemma:EM}, a nontrivial slice summand $\Sigma^{t-s+t\alpha}\EM E_2^{s,2t}(\BP(p))$ can only contribute to $E_1^{m,n,t}(p)$ when $m-t+s\ge 0$ and $n-t\le 0$.  (\aref{lemma:EM} actually provides a more stringent vanishing condition, but this ``non-fourth quadrant'' vanishing is all we need here.)  It follows that $E_1^{m,n,t}(p) = 0$ for $(p-2)n>(p-1)m$.  We conclude that $\ul\pi_{m+n\alpha}\hat\sphere_{(p)} = 0$ in this range as well.
\end{proof}

\section{Vanishing for the integral and $p$-local spheres}\label{sec:uncomp}

In this section, we study the problem of lifting vanishing results about $\ul\pi_\star\hat\sphere$ to $\ul\pi_\star\sphere$.  We first recall some base change theorems and use them to prove \aref{thm:integral0}.  We then recall Bachmann's theorem on $\SHA(F)[1/2,1/\eta]$, and finally prove \aref{thm:integral1}.

\subsection{Base change}\label{subsec:base}
Recall that for any map of schemes $f:S\to T$ one has a pullback, \emph{i.e.,}~base change, functor $f^*:\SHA(T)\to \SHA(S)$.  In this subsection, we use standard arguments with base change functors to expand the class of fields for which various vanishing results will hold.  We write $\sphere_F$ for the sphere spectrum in $\SHA(F)$ or $\SHA(F)[1/q]$.  Note that if $f:\Spec E\to \Spec F$ is an extension of fields, then $f^*\sphere_F = \sphere_E$.

The functor $f^*$ always admits a right adjoint $f_*$.  If $f$ is smooth, it also admits a left adjoint $f_\sharp$ (given by composition of the structure map to $S$ with $f$).  (See \cite[Appendix A]{Hoyois} for a brief review.)

Below we will use the fact that when $F$ is perfect, $\ul\pi_{m+n\alpha}X$ is a strictly $\AA^1$-invariant sheaf in the sense of \cite{morel:A1}.  (See \cite[\S1.2]{Hoyois} for a brief review.)

\begin{lemma}\label{lemma:ess}
Let $F$ be a perfect field and suppose $E/F$ is an essentially smooth field extension.  
Then there is an isomorphism $\ul\pi_\star\sphere_F(\Spec E)\cong \ul\pi_\star\sphere_E(\Spec E)$.
\end{lemma}
\begin{proof}
Write $\Spec E$ as a cofiltered limit $\lim_\beta X_\beta$ of smooth $F$-schemes $X_\beta$.  Fix $m,n\in \ZZ$.  We have
\[\begin{aligned}
  \ul\pi_{m+n\alpha}\sphere_F(\Spec E)
  &=\colim_\beta\ul\pi_{m+n\alpha}\sphere_F(X_\beta) && \text{(by definition)}\\
  &= \colim_\beta [\Sigma^{m+n\alpha}X_{\beta+},\sphere_F]\\
  &= \colim_\beta [f_{\beta\sharp} f_\beta^*\Sigma^{m+n\alpha}\sphere_F,\sphere_F]\\
  &\cong \colim_\beta [f_\beta^*\Sigma^{m+n\alpha}\sphere_F,f_\beta^*\sphere_F]\\
  &\cong [\Sigma^{m+n\alpha}\sphere_E,f^*\sphere_F] && \text{(by \cite[Lemma A.7(1)]{Hoyois})}\\
  &= \ul\pi_{m+n\alpha}\sphere_E(\Spec E),
\end{aligned}
\]
as desired.
\end{proof}

In the following proposition, we write $\pi_{m+n\alpha}X := \ul\pi_{m+n\alpha}X(\Spec F)$ for the $(m+n\alpha)$-th homotopy group (as opposed to sheaf) of $X\in \SHA(F)$.

\begin{prop}\label{prop:colim}
Suppose $F$ is a filtered colimit of fields $F = \colim_\beta F_\beta$.  If $\pi_{m+n\alpha}\sphere_{F_\beta} = 0$ for all $\beta$, then $\pi_{m+n\alpha}\sphere_F = 0$.
\end{prop}
\begin{proof}
This follows from \cite[Lemma A.7(1)]{Hoyois}.
\end{proof}

\begin{lemma}\label{lemma:cd}
Fix $m,n\in \ZZ$ and suppose $\ul\pi_{m+n\alpha}\sphere_k = 0$ for all nonreal characteristic $0$ fields $k$ with $\cd k<\infty$.  Then $\ul\pi_{m+n\alpha}\sphere_F = 0$ for any nonreal characteristic $0$ field $F$, regardless of cohomological dimension.
\end{lemma}
\begin{proof}
Since $\ul\pi_{m+n\alpha}\sphere_F$ is strictly $\AA^1$-invariant, it suffices to check that $\ul\pi_{m+n\alpha}\sphere_F(\Spec L) = 0$ for all finitely generated field extensions $L/F$.  By \aref{lemma:ess}, this is the same as showing $\pi_{m+n\alpha}\sphere_L = 0$.  Thus, by \aref{prop:colim} and our hypothesis, it suffices to show that $L$ is a filtered colimit of nonreal fields with finite cohomological dimension.

Since $L$ is nonreal, there exist $a_1,\ldots,a_n\in L$ such that $-1 = a_1^2+\cdots+a_n^2$.  Let $L_0 = \QQ(a_1,\ldots,a_n)$.  Then $L_0$ is nonreal and $\cd L_0\le 2$ by \cite[\S II.4.4, Proposition 13]{serre:gc}.  Letting $A$ range over finite subsets of $L\smallsetminus L_0$ we see that $L = \colim_A L_0(A)$.  By \cite[\S II.4.1 Proposition 10' \& \S II.4.2 Proposition 11]{serre:gc}, we see that $\cd L_0(A) < \infty$, and it is clear that each $L_0(A)$ is nonreal, completing our proof.
\end{proof}

We can now use \aref{thm:cdfinite} to prove \aref{thm:integral0}.

\begin{proof}[Proof of \aref{thm:integral0}]
The positive characteristic statement is a direct consequence of \aref{thm:cdfinite}.  Suppose $F$ is nonreal with characteristic $0$.  If $\cd F<\infty$, then \aref{thm:cdfinite} implies $\sphere\simeq\hat\sphere$, and \aref{thm:main} and \aref{thm:plocal} guarantee that $\ul\pi_\star \sphere$ has the stated vanishing range.  By \aref{lemma:cd}, this is enough to conclude that the same vanishing range holds when $\cd F = \infty$.
\end{proof}

We conclude this section with a virtual cohomological dimension version of \aref{lemma:cd}.  We will use it in \aref{subsec:integral1} to prove \aref{thm:integral1}.

\begin{lemma}\label{lemma:vcd}
Fix $m,n\in\ZZ$ and suppose $\ul\pi_{m+n\alpha}\sphere_k = 0$ for all formally real fields $k$ with $\vcd_2 k <\infty$ and all nonreal characteristic $0$ fields.  Then $\ul\pi_{m+n\alpha}\sphere_F = 0$ for any formally real field $F$, regardless of virtual $2$-cohomological dimension.
\end{lemma}
\begin{proof}
Suppose $F$ is a formally real field and $L$ is a finitely generated extension of $F$.  As in the proof of \aref{lemma:cd}, it suffices to show $\pi_{m+n\alpha}\sphere_L = 0$.  If $L$ is nonreal, we are done by hypothesis, so we may assume $L$ is formally real.  We aim to express $L$ as a filtered colimit of formally real fields with finite $\vcd_2$.  Letting $A$ range over finite subsets of $L\smallsetminus \QQ$, we see that $L = \colim_A \QQ(A)$.  Each $\QQ(A)$ is formally real (since it is a subfield of $L$) and since $\vcd_2 \QQ = 2 <\infty$, the same two propositions from \cite{serre:gc} imply that $\vcd_2 \QQ(A)<\infty$ for each $A$.  By \aref{prop:colim}, we are done.
\end{proof}

\subsection{Bachmann's theorem}\label{subsec:bachmann}

Let $\rho$ denote the map $\sphere\to \Sigma^\alpha\sphere$ induced by taking the non-basepoint of $S^0$ to $-1\in \AA^1\setminus 0$.  In \cite{bachmann:rho}, Bachmann finds an alternate presentation of the $\rho$-inverted stable motivic homotopy category $\SHA(F)[1/\rho]$ in terms of the real \'etale topology.  We will not go into the details of the real \'etale topology, instead sending the reader to \cite{ret}, especially its first chapter.  Let $(\Spec F)_\ret$ denote the site of \'etale schemes over $\Spec F$ with the real \'etale topology, and let $\SH(\Shv((\Spec F)_\ret))$ denote the local stable homotopy category of sheaves of spectra on $(\Spec F)_\ret$ (see \cite[\S 2]{bachmann:rho}).  Let $X_F$ denote the Harrison space of orderings on $F$.  By \cite[Theorem 1.3]{ret}, $\Shv(X_F)\simeq \Shv((\Spec F)_\ret)$.  Bachmann's theorem (specialized to the case in which the base scheme is $\Spec$ of a field) then tells us the following.

\begin{theorem}[{\cite[Theorem 31]{bachmann:rho}}]\label{thm:bachmann}
There are triangulated equivalences of categories
\[
  \SHA(F)[1/\rho]\simeq \SH(\Shv((\Spec F)_\ret))\simeq \SH(\Shv(X_F)).
\]
\end{theorem}

Because of the relation $(2+\rho\eta)\eta = 0$, we see that $\rho$ is invertible whenever $2$ and $\eta$ are invertible.  From this, Bachmann derives the following corollary.

\begin{corollary}\label{cor:bachmann}
There are triangulated equivalences of categories
\[
  \SHA(F)[1/2,1/\eta] \simeq \SH(\Shv((\Spec F)_\ret))[1/2]\simeq \SH(\Shv(X_F))[1/2].
\]
\end{corollary}

\begin{remark}
Note that when $X_F = *$, we have $\SH(\Shv(X_F)) = \SH$, the classical Spanier-Whitehead category.  When $X_F=*$ and $F$ admits a real embedding, the equivalences in \aref{thm:bachmann} and \aref{cor:bachmann} come from the real Betti realization functor \cite[\S 9]{bachmann:rho}.
\end{remark}

\subsection{Uncompletion}\label{subsec:integral1}

We now use Bachmann's theorem and several fracture squares to prove \aref{thm:integral1}.

\begin{proof}[Proof of \aref{thm:integral1}]
We prove the vanishing statement for $\ul\pi_\star\sphere$; the reader may check that an analogous argument easily covers the $p$-local version.

By \aref{lemma:vcd}, it suffices to show that $\ul\pi_\star \sphere_k$ obtains the stated vanishing range for all $k$ formally real with $\vcd_2 k<\infty$ or nonreal of characteristic $0$.  The vanishing range in \aref{thm:integral1} is a subset of the range from \aref{thm:integral0}, so the latter case is covered.  Now suppose $k$ is formally real with $\vcd_2 k<\infty$.  We claim that it suffices to check that the homotopy groups $\pi_\star\sphere_k$ obtain the vanishing range.  Indeed, if $E/k$ is a finitely generated field extension, \aref{lemma:ess} implies that $\ul\pi_{m+n\alpha}\sphere_k(\Spec E) = \pi_{m+n\alpha}\sphere_E$.  We either have that $E$ is nonreal of characteristic $0$ (and can invoke \aref{thm:integral0}), or that $E$ is formally real.  In the latter case, the results of \cite[\S II.4.1 \& II.4.2]{serre:gc} imply that $\vcd_2 E<\infty$, and we are still working with a homotopy group over a field satisfying our hypotheses.  Thus we have successfully reduced the problem to checking the vanishing range of $\pi_\star\sphere_k$ for $k$ formally real with $\vcd_2 k<\infty$.

Fix $F$ formally real with finite $\vcd_2$ and consider the following three homotopy pullback squares:
\[\xymatrix{
  \sphere\ar[r]\ar[d] &\hat\sphere\ar[d]\\
  \eta^{-1}\sphere\ar[r] &\eta^{-1}\hat\sphere\text{,}
}\qquad
\xymatrix{
  \eta^{-1}\sphere\ar[r]\ar[d] &\eta^{-1}\sphere\comp{2}\ar[d]\\
  \eta^{-1}\sphere[1/2]\ar[r] &\eta^{-1}\sphere\comp{2}[1/2]\text{,}
}\qquad
\xymatrix{
  \sphere\comp{2}\ar[r]\ar[d] &\sphere\comp{2,\eta}\ar[d]\\
  \eta^{-1}\sphere\comp{2}\ar[r] &\eta^{-1}\sphere\comp{2,\eta}.
}\]
The first is the $\eta$-primary fracture square for $\sphere$, the second is the $\eta$-periodization of the $2$-primary fracture square for $\sphere$, and the third is the $\eta$-primary fracture square for $\sphere\comp{2}$.  The vanishing ranges for $\pi_\star \hat\sphere$ and $\pi_\star \eta^{-1}\hat\sphere$ follow from \aref{thm:main}, so the first square implies that it suffices to check the vanishing range for $\pi_\star \eta^{-1}\sphere$.

This brings us to the second square.  We analyze the bottom row using \aref{cor:bachmann}, which tells us that $\pi_{m+n\alpha}\eta^{-1}\sphere[1/2] = 0$ if and only if $\sphere[1/2]\in \SH(\Shv(X_F))$ has $0$ as its $m$-th homotopy group.  We claim that this latter condition is obtained if and only if $\pi_m^{\top}\sphere[1/2] = 0$.  By the argument of \cite[Proposition 40]{bachmann:rho}, it suffices to check this condition when $F$ is real closed.  But then $X_F = *$ and $\SH(\Shv(X_F)) = \SH$, which is precisely the category in which $\pi_m^\top$ is computed.  By the same argument, $\pi_{m+n\alpha}\eta^{-1}\sphere\comp{2}[1/2] = 0$ if and only if $\pi_m^\top\sphere\comp{2}[1/2] = 0$.  (Note that by Serre finiteness, the set of such $m$ is a subset of those for which $\pi_m^\top\sphere[1/2] = 0$.)

It remains to check the vanishing range for $\pi_\star \eta^{-1}\sphere\comp{2}$.  Since $\vcd_2(F)<\infty$, \cite[Theorem 1]{MASSConv} implies that the top row of the third square is a $\pi_\star$-isomorphism, whence the bottom row is a $\pi_\star$-isomorphism as well.  In particular, $\pi_{m+n\alpha}\eta^{-1}\sphere\comp{2} = 0$ if and only if $\pi_{m+n\alpha}\eta^{-1}\sphere\comp{2,\eta} = 0$.  By \aref{thm:main}, this condition holds whenever $m<0$ or $m>0$, $m\equiv1$ or $2\pmod 4$, and $2n>\max\{3m+5,4m\}$.  It follows that $\pi_{m+n\alpha}\eta^{-1}\sphere = 0$ whenever $m<0$ or $m>0$, $m\equiv1$ or $2\pmod 4$,  $2n>\max\{3m+5,4m\}$, and $\pi_m^\top \sphere[1/2] =0$.  This concludes our proof.
\end{proof}

\section{Questions}\label{sec:q}
Here we present several natural questions raised by our work, along with some commentary.

\begin{question}
Given $m\in \ZZ$ such that the $m$-th $\eta$-complete Milnor-Witt stem $\ul\pi_{m+*\alpha}\hat\sphere$ is bounded above, what is the smallest $n\in \ZZ$ such that $\ul\pi_{m+n\alpha}\hat\sphere =0$?  If the $m$-th Milnor-Witt stem $\ul\pi_{m+*\alpha}\sphere$ is bounded above, what is the smallest $n$ such that $\ul\pi_{m+n\alpha}\sphere=0$?
\end{question}

\begin{remark}
The bounds presented here are far from optimal.  For instance, by \cite{RSO:pi1}, $\ul\pi_{1+3\alpha}\sphere = 0$, but the vanishing region of \aref{thm:integral1} is only obtained for $\ul\pi_{1+n\alpha}\sphere$ when $n>4$.  From the perspective of the slice spectral sequence, we lack both total information about the Novikov $E_2$-page and all the differentials in the spectral sequence.  While improvements on the vanishing range are no doubt possible via more nuanced slice arguments, it seems likely that different arguments would have to be invoked in order to find optimal bounds.
\end{remark}

Recall that $\omega\Ff$ denotes the contraction of a sheaf $\Ff$ \cite{morel:A1}.
\begin{question}
What constraints do the equations $\ul\pi_{m+n\alpha}\sphere = 0$ and $\omega\ul\pi_{m+(n-1)\alpha}\sphere \cong \ul\pi_{m+n\alpha}\sphere$ place on $\ul\pi_{m+(n-1)\alpha}\sphere$?
\end{question}

\begin{remark}
The above question is stated for homotopy sheaves of the sphere spectrum, but it could also be asked for abstract homotopy modules (objects in the heart of Morel's homotopy $t$-structure).  Matthias Wendt pointed out to the authors \cite[Lemma 3.7]{asokfasel} implies that $\omega\ul K_3^{\mathrm{ind}} = 0$, where $\ul K_3^{\mathrm{ind}}$ is the $3$-rd indecomposable $K$-sheaf, so it does not follow from abstract principles that the ``top'' sheaf in a Milnor-Witt stem is constant. Note that $\ul K_3^{\mathrm{ind}}$ coincides with the sheaf of integral motivic cohomology groups $\underline{H}^{1,2}$, which is non-zero and non-constant. For example, $H^{1,2}(\mathbb{Q})\iso\ZZ/24$ and $H^{1,2}(\mathbb{Q}(\sqrt{-1}))$ contains $\Z$ as a direct summand, see e.g., \cite[p.~542, 564]{merkurjev-suslin.k3}. All known top sheaves in Milnor-Witt stems are constant.
\end{remark}

We conclude by noting that the methods of \aref{sec:eta} imply another result, whose proof we only sketch.

\begin{theorem}\label{thm:etainv}
The natural map $\sphere\to\eta^{-1}\sphere$ induces an isomorphism $\ul\pi_{m+n\alpha}\sphere\cong\ul\pi_{m+n\alpha}\eta^{-1}\sphere$ whenever
\begin{itemize}
\item $m<0$, or
\item $m\ge 0$ and $2n>\max\{3m+5,4m\}$.
\end{itemize}
\end{theorem}
\begin{proof}[Proof Sketch]
By the $\eta$-primary fracture square, it suffices to prove the analogous result for $\hat\sphere\to\eta^{-1}\hat\sphere$.  

We compare the weight $n$ slice spectral sequence for $\sphere$ to the  weight $n$ $\eta$-periodic slice spectral sequence.  We get an isomorphism on $E_1$-pages above the slice-to-Novikov shift of the Andrews-Miller region from \aref{lemma:AM}.  This is precisely the region stated in the theorem.
\end{proof}

\begin{remark}
By Morel's computation of $\ul\pi_{n\alpha}\sphere$, the isomorphism in fact holds for $n>0$ when $m=0$.
\end{remark}

\begin{question}
For $m>0$, $2n>\max\{3m+5,4m\}$, and $n\equiv 0$ or $3\pmod 4$, what is $\ul\pi_{m+n\alpha}\sphere$?  What about $\ul\pi_{m+n\alpha}\hat\sphere$?
\end{question}
\begin{remark}
By \aref{thm:etainv}, this is equivalent to computing the homotopy sheaves of the $\eta$-periodic sphere spectrum.  
The global sections of these sheaves are computed for $F = \CC$ in \cite{AM}, and, for $F=\RR$, the $2$-complete global section computation appears in \cite{GI:etaR}.  
Calculations for $p$-adic fields $\QQ_{p}$ and the rational numbers $\QQ$ will appear in \cite{Wilson}.
Any sheaf computations and computations over a general field are completely open, except that we know we have vanishing under the conditions of \aref{thm:integral1}.

Note that the Andrews-Miller computation in \cite{AM} has nonzero groups for $m\equiv 0$ and $3\pmod{4}$, so the restrictions on $m$ in our vanishing theorems are necessary.
\end{remark}

\subsection*{Acknowledgements}
The authors gratefully acknowledge hospitality and support by the Mathematisches Forschungsinstitut Oberwolfach during Summer 2016 and the 
Institut Mittag-Leffler during Spring 2017. The first author was supported by NSF award DMS-1406327, the second author was supported by
the DFG priority programme ``Homotopy theory and algebraic geometry,''
and the third author was supported by the RCN 
Frontier Research Group Project no.~250399
``Motivic Hopf equations.''

\bibliographystyle{plain}
\bibliography{vanishing}

\end{document}